\pdfoutput=1
\RequirePackage{ifpdf}
\ifpdf 
\documentclass[pdftex]{sigma}
\else
\documentclass{sigma}
\fi

\numberwithin{equation}{section}

\newtheorem{Theorem}{Theorem}[section]
\newtheorem{Corollary}[Theorem]{Corollary}
\newtheorem{Lemma}[Theorem]{Lemma}
 { \theoremstyle{definition}
\newtheorem{Definition}[Theorem]{Definition}
\newtheorem{Example}[Theorem]{Example} }

\begin{document}

\allowdisplaybreaks

\newcommand{\arXivNumber}{1706.02391}

\renewcommand{\thefootnote}{}

\renewcommand{\PaperNumber}{085}

\FirstPageHeading

\ShortArticleName{The Inverse Spectral Problem for Jacobi-Type Pencils}

\ArticleName{The Inverse Spectral Problem for Jacobi-Type Pencils\footnote{This paper is a~contribution to the Special Issue on Orthogonal Polynomials, Special Functions and Applications (OPSFA14). The full collection is available at \href{https://www.emis.de/journals/SIGMA/OPSFA2017.html}{https://www.emis.de/journals/SIGMA/OPSFA2017.html}}}

\Author{Sergey M.~ZAGORODNYUK}

\AuthorNameForHeading{S.M.~Zagorodnyuk}

\Address{School of Mathematics and Computer Sciences, V.N.~Karazin Kharkiv National University,\\ Svobody Square 4, Kharkiv 61022, Ukraine}
\Email{\href{mailto:Sergey.M.Zagorodnyuk@gmail.com}{Sergey.M.Zagorodnyuk@gmail.com}}

\ArticleDates{Received June 10, 2017, in f\/inal form October 24, 2017; Published online October 28, 2017}

\Abstract{In this paper we study the inverse spectral problem for Jacobi-type pencils. By a Jacobi-type pencil we mean the following pencil $J_5 - \lambda J_3$, where $J_3$ is a Jacobi matrix and $J_5$ is a semi-inf\/inite real symmetric f\/ive-diagonal matrix with positive numbers on the second subdiagonal. In the case of a special perturbation of orthogonal polynomials on a~f\/inite interval the corresponding spectral function takes an explicit form.}

\Keywords{operator pencil; recurrence relation; orthogonal polynomials; spectral function}

\Classification{42C05; 47B36}

\renewcommand{\thefootnote}{\arabic{footnote}}
\setcounter{footnote}{0}

\section{Introduction}

The theory of operator pencils in Banach spaces has many applications and it is actively deve\-lo\-ping nowadays, see~\cite{cit_7000_Markus,cit_Piv,cit_8000_Rodman}. By operator pencils or operator polynomials one means polynomials of complex variable $\lambda$ whose coef\/f\/icients are linear bounded operators acting in a~Banach space~$X$
\begin{gather*} L(\lambda) = \sum_{j=0}^m \lambda^j A_j, \end{gather*}
where $A_j\colon X\rightarrow X$, $j=0,\dots,m$. As it was noted in Rodman's book~\cite{cit_8000_Rodman}, there exist three broad topics in this theory which are the main foci of interest:
\begin{itemize}\itemsep=0pt
\item[(a)] linearization, i.e., reduction to a linear polynomial;
\item[(b)] various types of factorizations;
\item[(c)] the problems of multiple completness of eigenvectors and generalized eigenvectors.
\end{itemize}
If $X$ is a Hilbert space and all operators $A_j$ are self-adjoint, the pencil $L(\lambda)$ is said to be self-adjoint. In the case $m=1$ ($m=2$) the pencil is called linear (respectively quadratic). For the recent progress on quadratic operator pencils and their applications we refer to the book of M\"oller and Pivovarchik~\cite{cit_Piv}.

Linear self-adjoint operator pencils has no spectral theory which could be compatible with the theory of single or commuting self-adjoint operators. The generalized eigenvalue problem
\begin{gather*} (A_0 + \lambda A_1) x = 0,\qquad x\in X, \end{gather*}
is essentially more generic than the usual eigenvalue problem for a self-adjoint operator. We can explain this dif\/ference by the following astonishing result from the theory of linear matrix pencils.

\begin{Theorem}[{\cite[Theorem 15.2.1, p.~341]{cit_7700_Parlett}}]\label{t1_1}
Any real square matrix $B$ can be written as \smash{$B = A M^{-1}$} or $B = M^{-1} A$ where~$A$ and~$M$ are suitable symmetric matrices.
\end{Theorem}

Nevertheless, the general spectral theory for linear operator pencils has old and new contributions. In particular, it has applications to quotients of bounded operators and to the study of electron waveguides in graphene, see~\cite{cit_7100_M_G_A} and references therein.

Recently, Ben Amara, Vladimirov and Shkalikov investigated the following linear pencil of dif\/ferential operators~\cite{cit_8_Ben_Amara_Vladimirov_Shkalikov}
\begin{gather}\label{f1_4}
(py'')'' -\lambda(-y'' + cr y)= 0.
\end{gather}
The initial conditions are $y(0) = y'(0) = y(1) = y'(1) = 0$, or $y(0) = y'(0) = y'(1) = (py'')'(1) + \lambda \alpha y(1) = 0$. Here $p,r\in C[0,1]$ are uniformly positive, while the parameters $c$ and $\alpha$ are real. Equation~(\ref{f1_4}) has several physical applications, including a motion of a partially f\/ixed bar with additional constraints in the elasticity theory~\cite{cit_8_Ben_Amara_Vladimirov_Shkalikov}.

It turns out that equation~(\ref{f1_4}) is related to certain generalized orthogonal polynomials, just as the Sturm--Liouville dif\/ferential operator is related to the orthogonal polynomials on the real line (OPRL).

The theory of orthogonal polynomials has numerous old and new contributions and applications, see~\cite{cit_3000_Chihara,cit_5000_Ismail, cit_10000_Simon, cit_20000_Suetin, cit_50000_Gabor_Szego,cit_70000_T}.
This theory is closely related to spectral problems for Jacobi matrices. The direct and inverse spectral problems for Jacobi matrices (with matrix elements) were described, e.g., in~\cite{cit_5_Atkinson,cit_10_Berezansky_Book}. For various other types of spectral problems we refer to historical notes in~\cite{cit_92000_Z_banded_matrices}. Recall the following basic def\/inition from~\cite{cit_95000_Z}.

\begin{Definition}\label{d1_1}A set $\Theta = (J_3, J_5, \alpha, \beta)$, where $\alpha>0$, $\beta\in\mathbb{R}$, $J_3$ is a Jacobi matrix and~$J_5$ is a semi-inf\/inite real symmetric f\/ive-diagonal matrix with positive numbers on the second subdiagonal, is said to be \textit{a Jacobi-type pencil $($of matrices$)$}.
\end{Definition}
From this def\/inition we see that matrices $J_3$ and $J_5$ have the following form
\begin{gather}
J_3 =
\left(
\begin{matrix}
b_0 & a_0 & 0 & 0 & 0 & \cdots\\
a_0 & b_1 & a_1 & 0 & 0 & \cdots\\
0 & a_1 & b_2 & a_2 & 0 &\cdots\\
\vdots & \vdots & \vdots & \ddots \end{matrix}
\right),\qquad a_k>0,\quad b_k\in\mathbb{R},\quad k\in\mathbb{Z}_+, \label{f1_5} \\
J_5 = \left(
\begin{matrix}
\alpha_0 & \beta_0 & \gamma_0 & 0 & 0 & 0 & \cdots\\
\beta_0 & \alpha_1 & \beta_1 & \gamma_1 & 0 & 0 & \cdots\\
\gamma_0 & \beta_1 & \alpha_2 & \beta_2 & \gamma_2 & 0 & \cdots\\
0 & \gamma_1 & \beta_2 & \alpha_3 & \beta_3 & \gamma_3 & \cdots \\
\vdots & \vdots & \vdots &\vdots & \ddots \end{matrix}
\right),\qquad \alpha_n,\beta_n\in\mathbb{R},\quad \gamma_n>0,\quad n\in\mathbb{Z}_+. \label{f1_10}
\end{gather}

With a Jacobi-type pencil of matrices $\Theta$ one associates a system of polynomials $\{ p_n(\lambda) \}_{n=0}^\infty$, such that
\begin{gather*} p_0(\lambda) = 1,\qquad p_1(\lambda) = \alpha\lambda + \beta, \end{gather*}
and
\begin{gather}\label{f1_20}
(J_5 - \lambda J_3) \vec p(\lambda) = 0,
\end{gather}
where $\vec p(\lambda) = (p_0(\lambda), p_1(\lambda), p_2(\lambda),\dots)^{\rm T}$. Here the superscript ${\rm T}$ means the transposition. Polynomials $\{ p_n(\lambda) \}_{n=0}^\infty$ are said to be \textit{associated to the Jacobi-type pencil of matrices~$\Theta$}.

Observe that for each system of orthonormal polynomials on the real line with $p_0=1$ one may choose $J_3$ to be the corresponding Jacobi matrix (which elements are the recurrence coef\/f\/icients), $J_5 = J_3^2$, and $\alpha$, $\beta$ being the coef\/f\/icients of~$p_1$ ($p_1(\lambda) = \alpha\lambda + \beta$).

One can rewrite relation~(\ref{f1_20}) in the scalar form
\begin{gather} \gamma_{n-2} p_{n-2}(\lambda) + (\beta_{n-1}-\lambda a_{n-1}) p_{n-1}(\lambda) + (\alpha_n-\lambda b_n) p_n(\lambda) \notag \\
\qquad{} + (\beta_n-\lambda a_n) p_{n+1}(\lambda) + \gamma_n p_{n+2}(\lambda) = 0,\qquad n\in\mathbb{Z}_+, \label{f1_30}
\end{gather}
where $p_{-2}(\lambda) = p_{-1}(\lambda) = 0$, $\gamma_{-2} = \gamma_{-1} = a_{-1} = \beta_{-1} = 0$.

Let us return to the dif\/ferential equation~(\ref{f1_4}) and explain how it is related to polynomials from~(\ref{f1_30}). The idea is the same as in the classical case. We consider a partition of the inter\-val~$[0,1]$
\begin{gather*} 0 = x_0 < x_1 < x_2 < \cdots < x_{N-1} < x_N = 1, \end{gather*}
with a uniform step $h$. Set $y_j := y(x_j)$, $p_j := p(x_j)$, $r_j := r(x_j)$, $0\leq j \leq N$. Then we replace the derivatives in~(\ref{f1_4}) by the corresponding discretizations $y'(x_j) \approx \frac{y_{j+1} - y_j}{h}$, $y''(x_j) \approx \frac{y_{j+1} - 2y_j + y_{j-1}}{h^2}$. We obtain the following equation
\begin{gather*} p_{j-1} y_{j-2} - 2(p_j + p_{j-1}) y_{j-1} + (p_{j+1}) + 4 p_j + p_{j-1}) y_j - 2(p_{j+1} + p_j) y_{j+1} \\
\qquad{} + p_{j+1} y_{j+2} - \lambda h^2 \big(y_{j-1} - \big(2 - h^2 c r_j\big) y_j + y_{j+1}\big)= 0.
\end{gather*}
Denote
\begin{gather*} \alpha_j = p_{j+1} + 4p_j + p_{j-1},\qquad \beta_j = -2(p_{j+1} + p_j),\qquad \gamma_j = p_{j+1}, \\
a_j = h^2,\qquad b_j = \big({-}2-h^2 c r_j\big) h^2.
\end{gather*}
Then we get
\begin{gather*} \gamma_{j-2} y_{j-2} + \beta_{j-1} y_{j-1} + \alpha_j y_j + \beta_j y_{j+1} + \gamma_j y_{j+2} + \lambda \big(
a_{j-1} y_{j-1} + b_j y_j + a_j y_{j+1}\big) = 0.
\end{gather*}
If we set $\widetilde\lambda = -\lambda$, we obtain a relation of the form~(\ref{f1_30}). Thus, relation~(\ref{f1_30}) forms a discrete grid model for the equation~(\ref{f1_4}). The boundary conditions $y(0)=0$, $y'(0)=0$ correspond to the convention $y_{-2}=y_{-1}=0$. Instead of $[0,1]$ we can consider $[0,+\infty)$. The conditions on the right end of the bar we replace by conditions on the left end which are controlled by $\alpha$ and $\beta$.

In this paper we shall introduce and study the direct and inverse spectral problems for Jacobi type pencils. We shall also consider a special perturbation of the class of orthonormal polynomials on a f\/inite real interval. In that case it is possible to obtain a more convenient integral representation for the corresponding spectral function.

{\bf Notations.} As usual, we denote by $\mathbb{R}$, $\mathbb{C}$, $\mathbb{N}$, $\mathbb{Z}$, $\mathbb{Z}_+$, the sets of real numbers, complex numbers, positive integers, integers and non-negative integers, respectively. By~$\mathbb{P}$ we denote the set of all polynomials with complex coef\/f\/icients.

By $l_2$ we denote the usual Hilbert space of all complex sequences $c = (c_n)_{n=0}^\infty = (c_0,c_1,c_2,\dots)^{\rm T}$ with the f\/inite norm $\| c \|_{l_2} = \sqrt{\sum\limits_{n=0}^\infty |c_n|^2}$. Here $T$ means the transposition. The scalar product of two sequences $c = (c_n)_{n=0}^\infty$, $d = (d_n)_{n=0}^\infty\in l_2$ is given by $(c,d)_{l_2} = \sum\limits_{n=0}^\infty c_n \overline{d_n}$. We denote $\vec e_m = (\delta_{n,m})_{n=0}^\infty\in l_2$, $m\in\mathbb{Z}_+$. By $l_{2,{\rm f\/in}}$ we denote the set of all f\/inite vectors from~$l_2$, i.e., vectors with all, but f\/inite number, components being zeros. By $\operatorname{diag} (c_1,c_2,c_3,\dots)$ we mean a~semi-inf\/inite matrix with $c_j\in\mathbb{C}$ in $j$-th column and $j$-th row and zeros outside the main diagonal.

By $\mathfrak{B}(\mathbb{R})$ we denote the set of all Borel subsets of $\mathbb{R}$. If $\sigma$ is a (non-negative) bounded measure on~$\mathfrak{B}(\mathbb{R})$ then by $L^2_\sigma$ we denote a Hilbert space of all (classes of equivalences of) complex-valued functions $f$ on $\mathbb{R}$ with a f\/inite norm $\| f \|_{L^2_\sigma} = \sqrt{ \int_\mathbb{R} |f(x)|^2 {\rm d}\sigma }$. The scalar product of two functions $f,g\in L^2_\sigma$ is given by $(f,g)_{L^2_\sigma} = \int_{\mathbb{R}} f(x) \overline{g(x)} {\rm d}\sigma$. By $[ f ]$ we denote the class of equivalence in~$L^2_\sigma$ which contains the representative~$f$. By $\mathcal{P}$ we denote a set of all (classes of equivalence which contain) polynomials in $L^2_\sigma$. As usual, we sometimes use the representatives instead of their classes in formulas. Let $B$ be an arbitrary linear operator in $L^2_\sigma$ with the domain $\mathcal{P}$. Let $f(\lambda)\in\mathbb{P}$ be nonzero and of degree $d\in\mathbb{Z}_+$, $f(\lambda) = \sum\limits_{k=0}^d d_k \lambda^k$, $d_k\in\mathbb{C}$. We set
\begin{gather*} f(B) = \sum_{k=0}^d d_k B^k,\qquad B^0 := E|_{\mathcal{P}}. \end{gather*}
If $f\equiv 0$, then $f(B):= 0|_{\mathcal{P}}$.

If H is a Hilbert space then $(\cdot,\cdot)_H$ and $\| \cdot \|_H$ mean the scalar product and the norm in $H$, respectively. Indices may be omitted in obvious cases. For a linear operator $A$ in $H$, we denote by $D(A)$ its domain, by $R(A)$ its range, by $\operatorname{Ker} A$ its null subspace (kernel), and $A^*$ means the adjoint operator if it exists. If $A$ is invertible then $A^{-1}$ means its inverse. $\overline{A}$ means the closure of the operator, if the operator is closable. If $A$ is bounded then $\| A \|$ denotes its norm. For a~set $M\subseteq H$ we denote by $\overline{M}$ the closure of $M$ in the norm of $H$. By $\operatorname{Lin} M$ we mean the set of all linear combinations of elements of $M$, and $\overline{\operatorname{span}}\, M := \overline{\operatorname{Lin} M}$. By $E=E_H$ ($0=0_H$) we denote the identity operator in $H$, i.e., $E_H x = x$, $x\in H$ (respectively the null operator in $H$, i.e., $0_H x = 0$, $x\in H$). If $H_1$ is a subspace of $H$, then $P_{H_1} = P_{H_1}^{H}$ is an operator of the orthogonal projection on $H_1$ in~$H$.

\section{Preliminaries}

In this section, for the convenience of the reader, we recall basic def\/initions and results from~\cite{cit_95000_Z}. Then we state the direct and inverse spectral problems for the Jacobi-type pencils. The direct problem will be solved immediately, while the inverse problem will be considered in the next section. Set
\begin{gather*}
u_n := J_3 \vec e_n = a_{n-1} \vec e_{n-1} + b_n \vec e_n + a_n \vec e_{n+1}, \\
w_n := J_5 \vec e_n = \gamma_{n-2} \vec e_{n-2} + \beta_{n-1} \vec e_{n-1} + \alpha_n \vec e_n + \beta_n \vec e_{n+1}+ \gamma_n \vec e_{n+2},\qquad n\in\mathbb{Z}_+.
\end{gather*}
Here and in what follows by $\vec e_k$ with negative $k$ we mean (vector) zero. The following operator
\begin{gather} A f = \frac{\zeta}{\alpha} (\vec e_1 - \beta \vec e_0) + \sum_{n=0}^\infty \xi_n w_n, \notag \\
f = \zeta \vec e_0 + \sum_{n=0}^\infty \xi_n u_n\in l_{2,{\rm f\/in}},\qquad \zeta, \xi_n\in\mathbb{C}, \label{f3_60}
\end{gather}
with $D(A) = l_{2,{\rm f\/in}}$ is said to be \textit{the associated operator for the Jacobi-type pencil $\Theta$}. Notice that in the sums in~(\ref{f3_60}) only f\/inite number of $\xi_n$ are nonzero. We shall always assume this in the case of elements from the linear span. In particular, we have
\begin{gather*} A J_3 \vec e_n = J_5 \vec e_n,\qquad n\in\mathbb{Z}_+, \end{gather*}
and therefore
\begin{gather*} A J_3 = J_5. \end{gather*}
As usual, the matrices $J_3$ and $J_5$ def\/ine linear operators with the domain $l_{2,{\rm f\/in}}$ which we denote by the same letters.

For an arbitrary non-zero polynomial $f(\lambda)\in\mathbb{P}$ of degree $d\in\mathbb{Z}_+$, $f(\lambda) = \sum\limits_{k=0}^d d_k \lambda^k$, $d_k\in\mathbb{C}$, we set $f(A) = \sum\limits_{k=0}^d d_k A^k$. Here $A^0 := E|_{l_{2,{\rm f\/in}}}$. For $f(\lambda) \equiv 0$, we set $f(A) = 0|_{l_{2,{\rm f\/in}}}$. The following relations hold~\cite{cit_95000_Z}
\begin{gather}\label{f3_110}
\vec e_n = p_n(A) \vec e_0,\qquad n\in\mathbb{Z}_+,\\
\label{f3_120} \big( p_n(A) \vec e_0, p_m(A) \vec e_0 \big)_{l_2} = \delta_{n,m},\qquad n,m\in\mathbb{Z}_+.
\end{gather}
Denote by $\{ r_n(\lambda) \}_{n=0}^\infty$, $r_0(\lambda) = 1$, the system of polynomials satisfying
\begin{gather}
\label{f3_130}
J_3 \vec r(\lambda) = \lambda \vec r(\lambda),\qquad \vec r(\lambda) = (r_0(\lambda), r_1(\lambda),r_2(\lambda), \dots )^{\rm T}.
\end{gather}
These polynomials are orthonormal on the real line with respect to a (possibly non-unique) non-negative f\/inite measure $\sigma$ on~$\mathfrak{B}(\mathbb{R})$ (Favard's theorem). Consider the following operator
\begin{gather}\label{f3_140}
U \sum_{n=0}^\infty \xi_n \vec e_n = \left[ \sum_{n=0}^\infty \xi_n r_n(x) \right ],\qquad \xi_n\in\mathbb{C},
\end{gather}
which maps $l_{2,{\rm f\/in}}$ onto $\mathcal{P}$. Here by $\mathcal{P}$ we denote a set of all (classes of equivalence which contain) polynomials in $L^2_\sigma$. Denote
\begin{gather} \label{f3_150}
\mathcal{A} = \mathcal{A}_\sigma = U A U^{-1}.
\end{gather}
The operator $\mathcal{A} = \mathcal{A}_\sigma$ is said to be \textit{the model representation in $L^2_\sigma$ of the associated operator $A$}.

\begin{Theorem}[\cite{cit_95000_Z}]\label{t3_1} Let $\Theta = (J_3, J_5, \alpha, \beta)$ be a Jacobi-type pencil. Let $\{ r_n(\lambda) \}_{n=0}^\infty$, $r_0(\lambda) = 1$, be a system of polynomials satisfying~\eqref{f3_130} and $\sigma$ be their $($arbitrary$)$ orthogonality measure on~$\mathfrak{B}(\mathbb{R})$. The associated polynomials $\{ p_n(\lambda) \}_{n=0}^\infty$ satisfy the following relations
\begin{gather}\label{f3_160}
\int_{\mathbb{R}} p_n(\mathcal{A})(1) \overline{ p_m(\mathcal{A})(1) } {\rm d}\sigma = \delta_{n,m},\qquad n,m\in\mathbb{Z}_+,
\end{gather}
where $\mathcal{A}$ is the model representation in $L^2_\sigma$ of the associated operator $A$.
\end{Theorem}

We have now recalled all basic notations and results which we shall need in the present paper. We introduce the following def\/inition.
\begin{Definition} \label{dd1_5} Let $\Theta = (J_3, J_5, \alpha, \beta)$ be a Jacobi-type pencil and $\{ p_n(\lambda) \}_{n=0}^\infty$ be the associated polynomials to $\Theta$. A sesquilinear functional $S(u,v)$, $u,v\in\mathbb{P}$, satisfying the following relation
\begin{gather*} S(p_n,p_m) = \delta_{n,m},\qquad n,m\in\mathbb{Z}_+, \end{gather*}
is said to be \textit{the spectral function of the Jacobi-type pencil $\Theta$}.
\end{Definition}

In a usual manner, we may introduce the corresponding direct and inverse spectral problems. \textit{The direct spectral problem for a Jacobi-type pencil~$\Theta$} consists in searching for answers on the following questions:
\begin{itemize}\itemsep=0pt
 \item[1)] Does the spectral function exist?
 \item[2)] If the spectral function exists, is it unique?
 \item[3)] If the spectral function exists, how to f\/ind it (or them)?
\end{itemize}

By relation~(\ref{f3_120}) we see that the spectral function always exists. By the linearity it follows that the spectral function is unique. It has the following representation
\begin{gather} \label{ff1_7}
S(u,v) = \big( u(A) \vec e_0, v(A) \vec e_0 \big)_{l_2},\qquad u,v\in\mathbb{P},
\end{gather}
where $A$ is the associated operator for $\Theta$. By~(\ref{ff1_7}) we obtain that
\begin{gather}\label{ff1_7_0}
\overline{S(u,v)} = S(v,u),\qquad u,v\in\mathbb{P},
\end{gather}
and
\begin{gather} \label{ff1_7_0_1}
S(u,u)\geq 0,\qquad u\in\mathbb{P}.
\end{gather}
Moreover, the following integral representation holds
\begin{gather} \label{ff1_7_1}
S(u,v) = \int_{\mathbb{R}} u(\mathcal{A}_\sigma)(1) \overline{ v(\mathcal{A}_\sigma)(1) } {\rm d}\sigma,\qquad u,v\in\mathbb{P}.
\end{gather}
Thus, the direct spectral problem for $\Theta$ is solved in full.

\textit{The inverse spectral problem for a Jacobi-type pencil $\Theta$} consists in searching for answers on the following questions:
\begin{itemize}\itemsep=0pt
\item[(a)] Is it possible to reconstruct the Jacobi-type pencil $\Theta = (J_3, J_5, \alpha, \beta)$ using its spectral function? If it is possible, what is the procedure of the reconstruction?
\item[(b)] What are necessary and suf\/f\/icient conditions for a sesquilinear functional $\sigma(u,v)$, $u,v\in\mathbb{P}$, to be the spectral function of a Jacobi-type pencil $\Theta = (J_3, J_5, \alpha, \beta)$?
\end{itemize}
Questions~(a) and (b) will be investigated in the next section.

\section{The inverse spectral problem for a Jacobi-type pencil}

An answer to the question~(b) of the inverse spectral problem is given by the following theorem.

\begin{Theorem}\label{tt2_1}A sesquilinear functional $S(u,v)$, $u,v\in\mathbb{P}$, satisfying relations~\eqref{ff1_7_0},~\eqref{ff1_7_0_1}, is the spectral function of a Jacobi-type pencil if and only if it admits the following integral representation
\begin{gather}\label{ff2_5}
S(u,v) = \int_{\mathbb{R}} u(\mathcal{A})(1) \overline{ v(\mathcal{A})(1) } {\rm d}\sigma,\qquad u,v\in\mathbb{P},
\end{gather}
where $\sigma$ is a non-negative measure on $\mathfrak{B}(\mathbb{R})$ with all finite power moments,
\begin{gather*} \int_\mathbb{R} {\rm d}\sigma = 1, \qquad \int_\mathbb{R} |g(x)|^2 {\rm d}\sigma > 0,
\end{gather*} for any non-zero complex polynomial~$g$, and $\mathcal{A}$ is a linear operator in~$L^2_\sigma$ with the following properties:
\begin{itemize}\itemsep=0pt
\item[$(i)$] $D(\mathcal{A}) = \mathcal{P}$;
\item[$(ii)$] for each $k\in\mathbb{Z}_+$ it holds
\begin{gather}\label{ff2_7}
\mathcal{A} x^k = \xi_{k,k+1} x^{k+1} + \sum_{j=0}^k \xi_{k,j} x^j,
\end{gather}
where $\xi_{k,k+1} > 0$, $\xi_{k,j}\in\mathbb{R}$, $0\leq j\leq k$;
\item[$(iii)$] the operator $\mathcal{A} \Lambda_0$ is symmetric, here by~$\Lambda_0$ we denote the operator of the multiplication by an independent variable in~$L^2_\sigma$ restricted to~$\mathcal{P}$.
\end{itemize}
\end{Theorem}
\begin{proof} \textit{Necessity.} Let $\Theta = ( J_3, J_5, \alpha, \beta )$ be a Jacobi-type pencil and $A$ be the associated operator of~$\Theta$. Def\/ine polynomials $\{ r_n(\lambda) \}$, the measure $\sigma$, the operators $U$ and $\mathcal{A}=\mathcal{A}_\sigma$ as in the introduction, see~(\ref{f3_130}), (\ref{f3_140}), (\ref{f3_150}). Let $S(u,v)$ be the spectral function of the pencil~$\Theta$. It has an integral representation~(\ref{ff1_7_1}). It remains to check that $\sigma$ and $\mathcal{A}$ possess all the required properties in the statement of the theorem.

Since $\sigma$ is generated by the Jacobi matrix $J_3$, it has all required properties as in the statement of the theorem. The operator~$\mathcal{A}_\sigma$ is linear and def\/ined on~$\mathcal{P}$.

\begin{Lemma}\label{ll2_1} The associated operator $A$ of a Jacobi-type pencil $\Theta$ has the following properties:
\begin{itemize}\itemsep=0pt
\item[$1)$] for each $n\in\mathbb{Z}_+$ it holds
\begin{gather} \label{ff2_10}
\vec e_n = d_{n,n} A^n \vec e_0 + \sum_{j=0}^{n-1} d_{n,j} A^j \vec e_0,
\end{gather}
where $d_{n,n} > 0$, $d_{n,j}\in\mathbb{R}$, $0\leq j\leq n-1$;

\item[$2)$] for each $n\in\mathbb{Z}_+$ it holds
\begin{gather}\label{ff2_12}
A^n \vec e_0 = c_{n,n} \vec e_n + \sum_{j=0}^{n-1} c_{n,j} \vec e_j,
\end{gather}
where $c_{n,n} > 0$, $c_{n,j}\in\mathbb{R}$, $0\leq j\leq n-1$;

\item[$3)$] for each $n\in\mathbb{Z}_+$ it holds
\begin{gather}\label{ff2_14}
A \vec e_n = f_{n,n+1} \vec e_{n+1} + \sum_{j=0}^{n} f_{n,j} \vec e_j,
\end{gather}
where $f_{n,n+1} > 0$, $f_{n,j}\in\mathbb{R}$, $0\leq j\leq n$.
\end{itemize}
\end{Lemma}

\begin{proof}[Proof of Lemma~\ref{ll2_1}] Let $p_n(\lambda) = \sum\limits_{j=0}^n d_{n,j} \lambda^j$, $d_{n,n}>0$, $d_{n,j}\in\mathbb{R}$, be the associated polyno\-mials of the pencil. By~(\ref{f3_110}) we conclude that relation~(\ref{ff2_10}) holds.

Let us check relation~(\ref{ff2_12}) by the induction argument. For the cases $n=0,1$ relation~(\ref{ff2_12}) obviously holds, see the def\/inition of $A$. Assume that~(\ref{ff2_12}) holds for $n=1,2,\dots ,r$, $r\in\mathbb{N}$. By~(\ref{ff2_10}) with $n=r+1$ we may write
\begin{gather*} A^{r+1} \vec e_0 = \frac{1}{d_{r+1,r+1}} \vec e_{r+1} - \frac{1}{d_{r+1,r+1}} \sum_{j=0}^{r} d_{r+1,j} A^j \vec e_0 \\
\hphantom{A^{r+1} \vec e_0}{}
= \frac{1}{d_{r+1,r+1}} \vec e_{r+1} + \sum_{j=0}^{r} \sum_{m=0}^{j} \frac{ (-1) d_{r+1,j} c_{j,m} }{d_{r+1,r+1}}\vec e_m.
\end{gather*}
Finally, in order to prove relation~(\ref{ff2_14}) we apply the operator $A$ to the both sides of relation~(\ref{ff2_10}) and use~(\ref{ff2_12})
\begin{gather*} A \vec e_n = d_{n,n} A^{n+1} \vec e_0 + \sum_{j=0}^{n-1} d_{n,j} A^{j+1} \vec e_0 \\
\hphantom{A \vec e_n}{}
= d_{n,n} c_{n+1,n+1} \vec e_{n+1} + \sum_{j=0}^{n} d_{n,n} c_{n+1,j} \vec e_j + \sum_{j=0}^{n-1} \sum_{k=0}^{j+1} d_{n,j} c_{j+1,k} \vec e_k,\qquad n\in\mathbb{Z}_+.
\end{gather*}
Lemma~\ref{ll2_1} is proved.
\end{proof}

Return to the proof of Theorem~\ref{tt2_1}. Choose an arbitrary $k\in\mathbb{Z}_+$. Observe that
\begin{gather*} x^k = \sum_{j=0}^k g_{k,j} r_j(x), \qquad g_{k,j}\in\mathbb{R},\quad g_{k,k}>0. \end{gather*}
By~(\ref{ff2_14}) we may write
\begin{gather} \mathcal{A}_\sigma \big[ x^k \big] = \sum_{j=0}^k g_{k,j} U A \vec e_j = \sum_{j=0}^k g_{k,j} \left(
f_{j,j+1} [r_{j+1}] + \sum_{l=0}^j f_{j,l} [r_l] \right) \notag \\
\hphantom{\mathcal{A}_\sigma \big[ x^k \big]}{} = g_{k,k} f_{k,k+1} [r_{k+1}] + \left[
\sum_{j=0}^{k-1} g_{k,j} f_{j,j+1} r_{j+1} + \sum_{j=0}^k g_{k,j} \sum_{l=0}^j f_{j,l} r_l\right] \notag \\
\hphantom{\mathcal{A}_\sigma \big[ x^k \big]}{}
= \left[ g_{k,k} f_{k,k+1} \widetilde{\mu_{k+1,k+1}} x^{k+1} + g_{k,k} f_{k,k+1} \sum_{m=0}^k \widetilde{\mu_{k+1,m}} x^{m} \right. \notag \\
\left.\hphantom{\mathcal{A}_\sigma \big[ x^k \big]=}{} +
\sum_{j=0}^{k-1} g_{k,j} f_{j,j+1} r_{j+1} + \sum_{j=0}^k g_{k,j} \sum_{l=0}^j f_{j,l} r_l \right], \label{ff2_15}
\end{gather}
where we set $r_l(\lambda) = \sum\limits_{j=0}^l \widetilde \mu_{l,j} \lambda^j$, $\widetilde\mu_{l,l}>0$, $\mu_{l,j}\in\mathbb{R}$. Since we get a~real polynomial of degree $k+1$ on the right-hand side of~(\ref{ff2_15}) and it has a positive leading coef\/f\/icient $g_{k,k} f_{k,k+1} \widetilde{\mu_{k+1,k+1}}$, then relation~(\ref{ff2_7}) is proved.

It remains to verify~$(iii)$. For arbitrary $n,m\in\mathbb{Z}_+$ we may write
\begin{gather}
(\mathcal{A}_\sigma\Lambda_0 [r_n(x)], [r_m(x)]) = \big(UAU^{-1} [a_{n-1} r_{n-1}(x) + b_n r_{n}(x) + a_n r_{n+1}(x)] , [r_m(x)]\big) \notag \\
\hphantom{(\mathcal{A}_\sigma\Lambda_0 [r_n(x)], [r_m(x)])}{}
 = (A J_3 \vec e_n, \vec e_m) = (J_5 \vec e_n, \vec e_m) = (\vec e_n, J_5 \vec e_m) \notag \\
\hphantom{(\mathcal{A}_\sigma\Lambda_0 [r_n(x)], [r_m(x)])}{}
= (\vec e_n, A (a_{m-1} \vec e_{m-1} + b_m \vec e_m + a_m \vec e_{m+1})) \nonumber\\
\hphantom{(\mathcal{A}_\sigma\Lambda_0 [r_n(x)], [r_m(x)])}{} = ([r_n(x)], \mathcal{A}_\sigma\Lambda_0 [r_m(x)]). \label{ff2_16}
\end{gather}
By the linearity we conclude that $\mathcal{A}_\sigma\Lambda_0$ is symmetric.

\textit{Sufficiency.} Suppose that a sesquilinear functional $S(u,v)$, $u,v\in\mathbb{P}$, satisfying relations~(\ref{ff1_7_0}), (\ref{ff1_7_0_1}), is given. Assume that it has the integral representation~(\ref{ff2_5}) where $\sigma$ is a non-negative probability measure on~$\mathfrak{B}(\mathbb{R})$ with all f\/inite moments, $\int_\mathbb{R} |g(x)|^2 {\rm d}\sigma > 0$, for any non-zero complex polynomial~$g$, and $\mathcal{A}$ is a linear operator in $L^2_\sigma$ with properties~$(i)$--$(iii)$. By condition~$(ii)$ and the induction argument it can be directly verif\/ied that for each $n\in\mathbb{Z}_+$ it holds
\begin{gather*} \mathcal{A}^n [1] = a_{n,n} [x^n] + \sum_{j=0}^{n-1} a_{n,j} \big[x^j\big], \end{gather*}
where $a_{n,n} > 0$, $a_{n,j}\in\mathbb{R}$, $0\leq j\leq n-1$.

Suppose that $S(u,u)=0$ for a complex polynomial $u$. Then
\begin{gather*} 0 = \| u(\mathcal{A})(1) \|^2_{L^2_\sigma}. \end{gather*}
Assume that $u$ is nonzero, $u(\lambda) = \sum\limits_{k=0}^n d_k \lambda^k$, $d_n\not=0$, $d_k\in\mathbb{C}$, $n\in\mathbb{Z}_+$. Observe that
\begin{gather*}
u(\mathcal{A})(1) = \left( \sum_{k=0}^n d_k \mathcal{A}^k \right) [1] = d_n \mathcal{A}^n [1] + \sum_{k=0}^{n-1} d_k \mathcal{A}^k [1] \\
\hphantom{u(\mathcal{A})(1)}{} = d_n \left( a_{n,n} \big[x^n\big] + \left[ \sum_{j=0}^{n-1} a_{n,j} x^j\right]\right)
+ \sum_{k=0}^{n-1} d_k \sum_{j=0}^k a_{k,j} \big[x^j\big] \\
\hphantom{u(\mathcal{A})(1)}{}
= \left[ d_n a_{n,n} x^n + d_n \sum_{j=0}^{n-1} a_{n,j} x^j + \sum_{k=0}^{n-1} \sum_{j=0}^k d_k a_{k,j} x^j \right] =: [r(x)].
\end{gather*}
We have a nonzero polynomial $r$ (of degree $n$) with $\| r \|_{L^2_\sigma} = 0$. This contradicts to our assumptions on the measure $\sigma$. Consequently, $u\equiv 0$.

The functional $S$ def\/ines an inner product on~$\mathbb{P}$. Thus, the complex vector space $\mathbb{P}$ becomes a space $H$ with a scalar product. It is a normed space with the norm $\| p \| = \sqrt{S(p,p)}$. We shall not need its completion. Set
\begin{gather*} \mathbf{p}_0(\lambda) = \frac{1}{ \| 1 \|_{H} } =1,\qquad
 \mathbf{p}_1(\lambda) = \frac{\lambda - S(x,1) 1}{ \| \lambda - S(x,1) 1 \|_{H} }. \end{gather*}
Let us check that $\mathbf{p}_1(\lambda)$ is well-def\/ined. Denote by $\{ s_k \}_{k=0}^\infty$ the power moments of $\sigma$
\begin{gather*} s_k = \int_\mathbb{R} x^k {\rm d}\sigma,\qquad k\in\mathbb{Z}_+. \end{gather*}
Let
\begin{gather*} \Delta_n := \det (s_{k+l})_{k,l=0}^n,\qquad n\in\mathbb{Z}_+,\qquad \Delta_{-1}:=1, \end{gather*}
be the corresponding Hankel determinants. Observe that
\begin{gather*} S(x,1) = \int \mathcal{A} (1) {\rm d}\sigma = \int (\xi_{0,1}\lambda + \xi_{0,0}) {\rm d}\sigma
= \xi_{0,1} s_1 + \xi_{0,0}\in\mathbb{R}, \\
 S(x,x) = \int \mathcal{A} (1) \overline{\mathcal{A} (1)} {\rm d}\sigma = \int (\xi_{0,1}\lambda + \xi_{0,0})^2 {\rm d}\sigma = \xi_{0,1}^2 s_2 + 2\xi_{0,1}\xi_{0,0} s_1 + \xi_{0,0}^2.
\end{gather*}
We may write
\begin{gather*} \| \lambda - S(x,1) 1 \|_{H}^2 = S(\lambda - S(x,1) 1, \lambda - S(x,1) 1) =
S(\lambda,\lambda) - (S(\lambda,1))^2 \\
 \hphantom{\| \lambda - S(x,1) 1 \|_{H}^2}{} = \xi_{0,1}^2 \big(s_2 - s_1^2\big) = \xi_{0,1}^2 \Delta_1 > 0,
\end{gather*}
where we have used our assumptions on the measure $\sigma$. Then
\begin{gather*} \mathbf{p}_1(\lambda) = \frac{\lambda - \xi_{0,1} s_1 - \xi_{0,0}}{ \xi_{0,1}\sqrt{\Delta_1} }. \end{gather*}
Set
\begin{gather}\label{ff2_20_5}
\alpha = \frac{1}{ \xi_{0,1}\sqrt{\Delta_1} },\qquad \beta = -\frac{\xi_{0,1} s_1 + \xi_{0,0}}{ \xi_{0,1}\sqrt{\Delta_1} }.
\end{gather}

Let $\{ \mathbf{r}_n(\lambda) \}_{n=0}^\infty$ be orthonormal polynomials with respect to the measure $\sigma$ (having positive leading coef\/f\/icients). Denote by $J_3$ the corresponding Jacobi matrix (formed by the recurrence coef\/f\/icients of $\mathbf{r}_n$). Denote
\begin{gather*} J_5 = (g_{m,n})_{m,n=0}^\infty,\qquad g_{m,n} := (\mathcal{A} \Lambda_0 [ \mathbf{r}_n(\lambda)], [ \mathbf{r}_m(\lambda)])_{L^2_\sigma}. \end{gather*}
By condition~$(iii)$ of the theorem we conclude that $J_5$ is a symmetric semi-inf\/inite matrix. Let
\begin{gather*} \mathbf{r}_n(\lambda) = \sum_{k=0}^n \eta_{n,k} \lambda^k,\qquad \eta_{n,k}\in\mathbb{R},\quad \eta_{n,n} > 0,\quad n\in\mathbb{Z}_+. \end{gather*}
For an arbitrary $n\in\mathbb{Z}_+$ by condition~$(ii)$ we may write
\begin{gather}\label{ff2_25}
\mathcal{A} \Lambda_0 [ \mathbf{r}_n(\lambda)] = \sum_{k=0}^n \eta_{n,k} \mathcal{A} \big[\lambda^{k+1}\big] = \big[
\eta_{n,n} \xi_{n+1,n+2} \lambda^{n+2} + d_{n+1}(\lambda) \big],
\end{gather}
where $d_{n+1}(\lambda)$ is a zero polynomial or a polynomial with real coef\/f\/icients, $\deg p\leq n+1$. Then
\begin{gather*} g_{m,n} = 0,\qquad m,n\in\mathbb{Z}_+, \quad m>n+2, \end{gather*}
and
\begin{gather*} g_{n+2,n} = \eta_{n,n} \xi_{n+1,n+2} \big(\big[\lambda^{n+2}\big], [r_{n+2}(\lambda)]\big) =
\frac{\eta_{n,n} \xi_{n+1,n+2}}{\eta_{n+2,n+2}} > 0,\qquad n\in\mathbb{Z}_+. \end{gather*}
By~(\ref{ff2_25}) we also see that $g_{m,n}$ are real numbers. We conclude that $J_5$ is real f\/ive-diagonal and it has positive numbers on the second sub-diagonal.

Consider a Jacobi-type pencil $\widetilde\Theta = (J_3, J_5, \alpha,\beta)$. Let $\{ p_n(\lambda) \}_{n=0}^\infty$ be the associated polynomials to $\widetilde\Theta$. For the pencil $\widetilde\Theta$ we def\/ine the standard objects from the introduction. The polynomials~$r_n(\lambda)$ (see~(\ref{f3_130})) coincide with $\mathbf{r}_n(\lambda)$. We may also choose the given $\sigma$ as the orthogonality measure (the orthogonality measure can be non-unique). The operators~$A$,~$U$ and~$\mathcal{A}_\sigma$ we def\/ine in the standard way, see~(\ref{f3_60}), (\ref{f3_140}), (\ref{f3_150}). We only can not use the brief notation~$\mathcal{A}$ for~$\mathcal{A}_\sigma$, since~$\mathcal{A}$ already denotes the operator in the integral representation of~$S$.

For the pencil $\widetilde\Theta$ we may apply our arguments in the proved Necessity. By relation~(\ref{ff2_16}) we may write
\begin{gather*} (\mathcal{A}_\sigma \Lambda_0 [r_n],[r_m])_{L^2_\sigma} = (J_5 \vec e_n, \vec e_m)_{l_2} = g_{m,n} = (\mathcal{A} \Lambda_0 [r_n],[r_m])_{L^2_\sigma},\qquad n,m\in\mathbb{Z}_+. \end{gather*}
Here the last equality follows by the def\/inition of the matrix $J_5$. Therefore
\begin{gather}\label{ff2_30}
\mathcal{A}_\sigma [ \lambda p(\lambda) ] = \mathcal{A} [ \lambda p(\lambda) ],\qquad \forall\, p\in\mathbb{P}.
\end{gather}
Notice that
\begin{gather}\label{ff2_32}
\mathcal{A}_\sigma [ 1 ] = U A \vec e_0 =\frac{1}{\alpha} U (\vec e_1 - \beta \vec e_0)=\left[ \frac{1}{\alpha} (r_1(\lambda) - \beta) \right].
\end{gather}
The orthonormal polynomial $r_1(\lambda)$ has the following form
\begin{gather}\label{ff2_34}
r_1(\lambda) = \frac{1}{\sqrt{\Delta_1}} (\lambda - s_1).
\end{gather}
By~(\ref{ff2_32}), (\ref{ff2_34}) and (\ref{ff2_20_5}) we conclude that
\begin{gather*} \mathcal{A}_\sigma [ 1 ] = [ \xi_{0,1} \lambda + \xi_{0,0} ]. \end{gather*}
On the other hand, by property $(ii)$ we have $\mathcal{A} [ 1 ] = [ \xi_{0,1} \lambda + \xi_{0,0} ]$. Therefore
\begin{gather}\label{ff2_40}
\mathcal{A}_\sigma [ 1 ] = \mathcal{A} [ 1 ].
\end{gather}
Relations~(\ref{ff2_30}) and~(\ref{ff2_40}) show that $\mathcal{A}_\sigma = \mathcal{A}$. Comparing relations~(\ref{ff1_7_1}) and~(\ref{ff2_5}) we see that the spectral function of $\widetilde\Theta$ coincides with~$S$.
\end{proof}

Theorem~\ref{tt2_1} provides characteristic properties for the model representation $\mathcal{A}$ of the asso\-ciated operator of a Jacobi-type pencil. It is seen that these properties are close to the properties of the multiplication operator $\Lambda_0$. Of course, $\Lambda_0$ itself satisf\/ies properties $(i)$--$(iii)$.

\begin{Corollary}\label{cc2_1} Let $\sigma$ be a non-negative measure on $\mathfrak{B}(\mathbb{R})$ with all finite power moments, $\int_\mathbb{R} {\rm d}\sigma = 1$, $\int_\mathbb{R} |g(x)|^2 {\rm d}\sigma > 0$, for any non-zero complex polynomial~$g$. A~linear operator~$\mathcal{A}$ in~$L^2_\sigma$ is a model representation in $L^2_\sigma$ of the associated operator of a Jacobi-type pencil if and only if properties $(i)$--$(iii)$ of Theorem~{\rm \ref{tt2_1}} hold.
\end{Corollary}
\begin{proof} It follows directly from our constructions in the proof of Theorem~\ref{tt2_1}.
\end{proof}

Let $\Theta=(J_3,J_5,\alpha,\beta)$ be a Jacobi-type pencil and $\mathcal{A}$ be a model representation in $L^2_\sigma$ of the associated operator of $\Theta$. By the latter corollary we conclude that $\mathcal{A} \Lambda_0$ is symmetric
\begin{gather}\label{ff2_45}
(\mathcal{A} \Lambda_0 [u(\lambda)], [v(\lambda)])_{L^2_\sigma} = ([u(\lambda)], \mathcal{A} \Lambda_0 [v(\lambda)])_{L^2_\sigma},\qquad
u,v\in\mathbb{P}.
\end{gather}
Suppose that the orthogonality measure $\sigma$ is supported inside a f\/inite real segment $[a,b]$, $0 < a < b < +\infty$, i.e.,
$\sigma(\mathbb{R}\backslash[a,b]) = 0$.
In this case the operator $\Lambda$ of the multiplication by an independent variable has a bounded inverse on the whole~$L^2_\sigma$. By~(\ref{ff2_45}) we may write
\begin{gather*} \big(\Lambda^{-1} \mathcal{A} [\lambda u(\lambda)], [\lambda v(\lambda)]\big)_{L^2_\sigma} =
\big(\Lambda^{-1} [\lambda u(\lambda)], \mathcal{A} [\lambda v(\lambda)]\big)_{L^2_\sigma},\qquad u,v \in\mathbb{P}. \end{gather*}
Denote $\mathcal{P}_0 = \Lambda \mathcal{P}$ and $\mathcal{A}_0 = \mathcal{A}|_{\mathcal{P}_0}$. Then
\begin{gather*} \big(\Lambda^{-1} \mathcal{A}_0 f,g\big)_{L^2_\sigma} = \big(\Lambda^{-1} f, \mathcal{A}_0 g\big)_{L^2_\sigma},\qquad f,g \in\mathcal{P}_0. \end{gather*}
Thus, in this case \textit{$\mathcal{A}_0$ is symmetric with respect to the form $(\Lambda^{-1}\cdot,\cdot)_{L^2_\sigma}$}. Analogous arguments were used in the theory of operator pencils, see~\cite[Chapter~IV, p.~163]{cit_7000_Markus}.

\begin{Example}\label{ee2_1} Let $a_n = 1$, $b_n = c$, $c>2$, $\alpha_n\in\mathbb{R}$, $\beta_n = 0$, $\gamma_n = 1$, $n\in\mathbb{Z}_+$. Def\/ine~$J_3$ and~$J_5$ by~(\ref{f1_5}),~(\ref{f1_10}) with the above parameters. Consider a Jacobi-type pencil $\Theta = (J_3,J_5,\alpha,\beta)$, with arbitrary $\alpha > 0$ and $\beta\in\mathbb{R}$. In this case
\begin{gather*} r_n(x) = U_n\left( \frac{x-c}{2} \right),\qquad n\in\mathbb{Z}_+, \end{gather*}
where $U_n(t) = \frac{\sin((n+1)\arccos t)}{ \sqrt{1-t^2} }$ is Chebyshev's polynomial of the second kind. The orthonormality relations for $r_n(x)$ have the following form
\begin{gather*} \int_{c-2}^{c+2} r_n(\lambda) r_m(\lambda) \sqrt{1 - \left( \frac{\lambda - c}{2} \right)^2} \frac{1}{\pi} {\rm d}\lambda
= \delta_{n,m},\qquad n,m\in\mathbb{Z}_+. \end{gather*}
By the recurrence relation~(\ref{f1_30}) we calculate
\begin{gather*} p_2(\lambda) = \alpha \lambda^2 + (c+\beta) \lambda - \alpha_0. \end{gather*}
Since $(c+\beta)^2 + 4 \alpha\alpha_0$ can be made zero or negative by a proper choice of~$\alpha_0$, then we have a~non-classical case. Thus, this case is worthy of an additional investigation which will be done elsewhere.
\end{Example}

We can state the following \textit{moment problem for Jacobi-type pencils}: f\/ind a non-negative measure $\sigma$ on $\mathfrak{B}(\mathbb{R})$ with all f\/inite power moments, $\int_\mathbb{R} {\rm d}\sigma = 1$, $\int_\mathbb{R} |g(x)|^2 {\rm d}\sigma > 0$, for any non-zero complex polynomial $g$, and a linear operator $\mathcal{A}$ in $L^2_\sigma$ with properties $(i)$--$(iii)$ of Theorem~\ref{tt2_1} such that
 \begin{gather*} \int_{\mathbb{R}} \mathcal{A}^m(1) \overline{ \mathcal{A}^n(1) } {\rm d}\sigma = s_{m,n},\qquad m,n\in\mathbb{Z}_+, \end{gather*}
where $\{ s_{m,n} \}_{m,n=0}^\infty$ is a prescribed set of real numbers (called moments).

As usual, there appear three important questions: 1)~the solvability of the moment problem; 2)~the uniqueness of a solution (the determinateness); 3)~a description of all solutions. This moment problem will be studied elsewhere.

\section[A special perturbation of orthogonal polynomials on a f\/inite interval]{A special perturbation of orthogonal polynomials \\ on a f\/inite interval}

If we look at the orthogonality relations~(\ref{f3_160}) or at the representation of the spectral function~(\ref{ff1_7_1}), we can see that we need to calculate a polynomial of the operator~$\mathcal{A}$. However, we do not have at hand any functional calculus for the operator~$\mathcal{A}$. It remains to calculate the powers of~$\mathcal{A}$ recurrently. It would be helpful to omit this procedure and to have a more transparent relation. It is possible to do this in the following special case.

Consider a Jacobi-type pencil $\Theta = (J_3,J_5,\alpha,\beta)$, with $J_3$, $J_5$ def\/ined by~(\ref{f1_5}),~(\ref{f1_10}), satisfying the following conditions
\begin{gather}
\sigma(\mathbb{R}\backslash[-c,c]) = 0,\qquad 0<c<1, \label{ff3_5} \\
J_5 = a J_3^2 + b J_3 + d \operatorname{diag} (1,0,0,0,\dots ),\qquad a>0,\quad b,d\in\mathbb{R}, \notag \\
\alpha = \frac{1}{a a_0},\qquad \beta = -\frac{b_0}{a_0} - \frac{b}{a a_0}. \label{ff3_15}
\end{gather}
Here $\sigma$, as before (see the introduction), is the orthogonality measure for polynomials $\{ r_n(x) \}_{n=0}^\infty$ (related to $J_3$).

Introduce other related objects from the introduction: the space $L^2_\sigma$, the operators~$A$,~$U$ and $\mathcal{A} = UAU^{-1}$. Observe that
\begin{gather*} U J_3 U^{-1} [p(x)] = [xp(x)],\qquad p\in\mathbb{P}. \end{gather*}
For an arbitrary $n\in\mathbb{Z}_+$ we may write
\begin{gather*} \mathcal{A} [x r_n(x)] = UAU^{-1} U J_3 U^{-1} [r_n(x)] = UAJ_3 \vec e_n = U J_5 \vec e_n \\
\hphantom{\mathcal{A} [x r_n(x)]}{} = U \big( a J_3^2 + b J_3 + d \operatorname{diag} (1,0,0,0,\dots ) \big) \vec e_n \\
\hphantom{\mathcal{A} [x r_n(x)]}{} = a U J_3 U^{-1} U J_3 \vec e_n + bUJ_3 \vec e_n + d U \operatorname{diag} (1,0,0,0,\dots ) \vec e_n \\
\hphantom{\mathcal{A} [x r_n(x)]}{} = \big[(ax+b)x r_n(x) + d (r_n,r_0)_{L^2_\sigma}\big].
\end{gather*}
By the linearity we get
\begin{gather*} \mathcal{A} [xp(x)] = \big[(ax+b)x p(x) + d (p,1)_{L^2_\sigma}\big],\qquad p\in\mathbb{P}. \end{gather*}
Moreover, we have
\begin{gather*} \mathcal{A}[1] = UAU^{-1} U\vec e_0 = UA\vec e_0 = U \frac{1}{\alpha} (\vec e_1 - \beta \vec e_0) = \left[\frac{1}{\alpha} (r_1(x) - \beta) \right].
\end{gather*}
Choose an arbitrary polynomial $q(x)\in\mathbb{P}$. We may write $q(x)= x \frac{q(x)-q(0)}{x} + q(0)$. Then
\begin{gather*} \mathcal{A} [q(x)] = \mathcal{A} \left[ x \frac{q(x)-q(0)}{x} \right] + \mathcal{A} [q(0)] \\
\hphantom{\mathcal{A} [q(x)]}{} = \left[\big(ax^2 +bx\big) \frac{q(x)-q(0)}{x} + d \left( \frac{q(x)-q(0)}{x}, 1 \right)_{L^2_\sigma} + \frac{q(0)}{\alpha} (r_1(x)-\beta) \right] \\
\hphantom{\mathcal{A} [q(x)]}{} = \left[ (ax+b) q(x) + d \left( \frac{q(x)-q(0)}{x}, 1 \right)_{L^2_\sigma} \right].
\end{gather*}
Denote
\begin{gather*} s_k := \int_{\mathbb{R}} x^k {\rm d}\sigma,\qquad k\in\mathbb{Z}_+,\qquad s_{-1}:=0. \end{gather*}
Observe that
\begin{gather*} \mathcal{A} \left[\sum_{k=0}^\infty c_k x^k \right] = \sum_{k=0}^\infty c_k \mathcal{A} \big[x^k\big]
 = \sum_{k=0}^\infty c_k \big[ ax^{k+1} + bx^k + d s_{k-1} \big],\qquad c_k\in\mathbb{C},
\end{gather*}
where all but f\/inite number of $c_k$ are zeros (i.e., $\mathcal{A}$ acts on polynomials).

Observe that the operator $\mathcal{A}$ can be unbounded: the corresponding example will be considered below. Thus, there appears a question: how to simplify the calculation of $u(\mathcal{A})$ for a complex polynomial $u$ in the integral representation~(\ref{ff1_7_1})? Is there any (at least) polynomial calculus for $\mathcal{A}$? The answer is af\/f\/irmative. Consider the following transformation from $L^2_\sigma$ to~$l_2$
\begin{gather*} G \left[ \sum_{k=0}^\infty c_k x^k \right] = \sum_{k=0}^\infty c_k \vec e_k,\qquad c_k\in\mathbb{C}. \end{gather*}
Here all but f\/inite number of $c_k$ are zeros. This will be assumed in what follows when dealing with operators on polynomials. Since $\sigma$ satisf\/ies conditions for the integral representation~(\ref{ff2_5}) (see the Necessity of the proof of Theorem~\ref{tt2_1}), then $\int_\mathbb{R} |g(x)|^2 {\rm d}\sigma > 0$, for any non-zero complex polynomial $g$. This shows that two dif\/ferent polynomials can not belong to the same class of the equivalence in $L^2_\sigma$. Thus, the operator $G$ is well-def\/ined. It is a linear operator with $D(G) = \mathcal{P}$ and $R(G) = l_{2,{\rm f\/in}}$. Moreover, $G$ is invertible, since the polynomial is uniquely determined by its coef\/f\/icients. Set
\begin{gather*} \mathbf{A} = G \mathcal{A} G^{-1}. \end{gather*}
Then $\mathbf{A}$ is a linear operator in $l_2$ with $D(\mathbf{A})=l_{2,{\rm f\/in}}$. Consider the following shift operator on the whole~$l_2$
\begin{gather*} S x = \sum_{k=0}^\infty c_k \vec e_{k+1},\qquad
x\in l_2,\quad x = (c_k)_{k=0}^\infty,\quad c_k\in\mathbb{C}. \end{gather*}
Notice that $S$ is linear and $\| Sx \| = \| x \|$, $\forall\, x\in l_2$. We may write
\begin{gather*} \mathbf{A} \left( \sum_{k=0}^\infty c_k \vec e_k \right) =
(a S + b E) \sum_{k=0}^\infty c_k \vec e_k + \left( d \sum_{k=0}^\infty c_k s_{k-1} \right) \vec e_0. \end{gather*}
Here all but f\/inite number of $c_k\in\mathbb{C}$ are zeros. This will be assumed in the sequel for operators on $l_{2,{\rm f\/in}}$. By condition~(\ref{ff3_5}) we see that
\begin{gather*} |s_n| = \left| \int_\mathbb{R} x^n \chi_{[-c,c]}(x) {\rm d}\sigma \right| \leq \int_\mathbb{R} \big| x^n \chi_{[-c,c]}(x) \big| {\rm d}\sigma
\leq c^n,\qquad n\in\mathbb{Z}_+, \end{gather*}
where $\chi_{[-c,c]}(x)$ is the characteristic function of the segment $[-c,c]$. Therefore $\vec s := \sum\limits_{k=1}^\infty s_{k-1} \vec e_k$ belongs to $l_2$. Consequently, we may write
\begin{gather*} \mathbf{A} w = (a S + b E) w + d (w,\vec s)_{l_2} \vec e_0,\qquad w\in l_{2,{\rm f\/in}}. \end{gather*}
Set
\begin{gather*} \widehat{\mathbf{A}} w = (a S + b E) w + d (w,\vec s)_{l_2} \vec e_0,\qquad w\in l_2. \end{gather*}
The linear operator $\widehat{\mathbf{A}}$ is an extension of $\mathbf{A}$. Observe that the operator $\widehat{\mathbf{A}}$ is \textit{bounded} and
\begin{gather*} \big\| \widehat{\mathbf{A}} \big\| \leq a + |b| + |d| \| \vec s \|_{l_2}. \end{gather*}
We can apply Riesz's calculus for $\widehat{\mathbf{A}}$. For an arbitrary polynomial $u\in\mathbb{P}$ we may write
\begin{gather} G \big(u\big(\mathcal{A}\big) [1]\big) = u\big(\widehat{\mathbf{A}}\big) \vec e_0
= -\frac{1}{2\pi i} \int_\gamma u(z) R_z\big(\widehat{\mathbf{A}}\big) {\rm d}z \cdot \vec e_0 =
-\frac{1}{2\pi i} \int_\gamma u(z) \big( R_z\big(\widehat{\mathbf{A}}\big) \vec e_0 \big) {\rm d}z,\label{ff3_53}
\end{gather}
where $\gamma$ is a circle centered at zero with a radius $\rho$ bigger then $a + |b| + |d| \| \vec s \|_{l_2}$, and $R_z\big(\widehat{\mathbf{A}}\big) = \big(\widehat{\mathbf{A}} - zE\big)^{-1}$. Here the last integral converges in the norm of~$l_2$.

Let us calculate $R_z(\widehat{\mathbf{A}}) \vec e_0 =: \vec f$, $z\in\gamma$. We may write
\begin{gather*} \vec e_0 = \big(\widehat{\mathbf{A}} - zE\big) \vec f = (aS+bE) \vec f + d \big(\vec f,\vec s\big)_{l_2} \vec e_0 - z\vec f. \end{gather*}
Then
\begin{gather*} \frac{1}{a} \big(1- d \big(\vec f,\vec s\big)_{l_2}\big) \vec e_0 = \left( S - \left(\frac{z-b}{a}\right) E \right) \vec f. \end{gather*}
Since $z\in\gamma$, then $\big| \frac{z-b}{a} \big| >1$ and
\begin{gather}\label{ff3_55}
\vec f = \frac{1- d (\vec f,\vec s)_{l_2}}{a} \left( S - \left(\frac{z-b}{a}\right) E \right)^{-1} \vec e_0.
\end{gather}
Denote $\vec u = \sum\limits_{k=0}^\infty u_k \vec e_k := \big( S - \big(\frac{z-b}{a}\big) E \big)^{-1} \vec e_0$. Then
\begin{gather*} \vec e_0 = \left( S - \left(\frac{z-b}{a}\right) E \right) \vec u =S \vec u + \left( \frac{b-z}{a} \right) \vec u. \end{gather*}
For the components of $\vec u$ we obtain the following equations
\begin{gather*} \left( \frac{b-z}{a} \right) u_0 =1,\qquad u_{n-1} + \left( \frac{b-z}{a} \right) u_n = 0,\qquad n\in\mathbb{N}. \end{gather*}
Then
\begin{gather*} \vec u = \sum_{k=0}^\infty (-1) \left( \frac{a}{z-b} \right)^{k+1} \vec e_k. \end{gather*}
By~(\ref{ff3_55}) we may write
\begin{gather*} \vec f = \tau \sum_{k=0}^\infty (-1) \left( \frac{a}{z-b} \right)^{k+1} \vec e_k,\qquad \tau\in\mathbb{C}, \end{gather*}
where
\begin{gather*} \tau = \frac{1- d \big(\vec f,\vec s\big)_{l_2}}{a}. \end{gather*}
Then
\begin{gather*} a \tau = 1- d \left(\tau \sum_{k=0}^\infty (-1) \left( \frac{a}{z-b} \right)^{k+1} \vec e_k,\vec s \right)_{l_2}, \end{gather*}
and
\begin{gather*} \tau = \frac{1}{a + d \sum\limits_{k=0}^\infty (-1)\big( \frac{a}{z-b} \big)^{k+1} s_{k-1}}. \end{gather*}
Denote
\begin{gather*} s(z) := \sum_{k=0}^\infty (-1)\left( \frac{a}{z-b} \right)^{k+1} s_{k-1},\qquad z\in\gamma,\\
\vec v(z) := \sum_{k=0}^\infty (-1)\left( \frac{a}{z-b} \right)^{k+1} \vec e_k,\qquad z\in\gamma. \end{gather*}
Then
\begin{gather*} R_z\big(\widehat{\mathbf{A}}\big) \vec e_0 = \vec f = \frac{1}{a+d s(z)} \vec v(z),\qquad z\in\gamma. \end{gather*}
Thus, we obtained a transparent expression for $R_z\big(\widehat{\mathbf{A}}\big) \vec e_0$. It can be used in relation~(\ref{ff3_53}) and it gives a transparent expression for $u(\mathcal{A}) [1]$.

\begin{Example}\label{ee3_1} Let $\Theta$ be a Jacobi-type pencil with $\alpha = \beta = \sqrt{2}$, $a_k = \sqrt{2}$, $b_k = 2$, $k\in\mathbb{Z}_+$, $\alpha_n = \beta_n = 0$, $\gamma_n = 1$, $n\in\mathbb{Z}_+$, and $J_3$, $J_5$ have form~(\ref{f1_5}),~(\ref{f1_10}). This Jacobi type pencil was considered in~\cite{cit_95000_Z} and explicit formulas for the associated polynomials $p_n(\lambda)$ were obtained.

Let $\kappa$ be an arbitrary positive number bigger than $2+2\sqrt{2}$. Consider a Jacobi-type pencil $\widehat\Theta = \big(\widehat J_3, \widehat J_5, \widehat \alpha, \widehat \beta\big)$, where
\begin{gather*} \widehat J_3 = \frac{1}{\kappa} J_3,\qquad \widehat J_5 = \frac{1}{\kappa} J_5,\qquad
\widehat\alpha = \alpha,\qquad \widehat\beta = \beta. \end{gather*}
Notice that the associated polynomials and the associated operators for $\Theta$ and $\widehat\Theta$ are the same. It was shown in~\cite{cit_95000_Z} that the associated operator for $\Theta$ is unbounded. Thus, the associated operator for the pencil $\widehat\Theta$ is unbounded, as well. The pencil $\widehat\Theta$ satisf\/ies conditions~(\ref{ff3_5}),~(\ref{ff3_15}) with
\begin{gather*} c= \frac{2+2\sqrt{2}}{\kappa},\qquad a = \frac{1}{2}\kappa,\qquad b = -2,\qquad d=\frac{1}{\kappa}. \end{gather*}
\end{Example}

\subsection*{Acknowledgements}
The author is grateful to referees for their valuable comments and suggestions which led to an essential improvement of the paper.

\pdfbookmark[1]{References}{ref}
\LastPageEnding

\end{document}